\documentclass[reqno]{amsart}
\usepackage{amsmath,amsthm,amssymb,hyperref,graphicx,mathrsfs,mathtools} 

\newtheorem{theorem}{Theorem}
\newtheorem{corollary}{Corollary}
\newtheorem{lemma}{Lemma}
\newtheorem{thmx}{Theorem}
\newtheorem{remark}{Remark}

\allowdisplaybreaks

\begin{document}
\title[INEQUALITIES CONCERNING RATIONAL FUNCTIONS WITH PRESCRIBED
POLES ]{INEQUALITIES CONCERNING RATIONAL FUNCTIONS WITH PRESCRIBED
POLES}
\author{N. A. Rather$^1$}
\author{Tanveer Bhat$^2$}
\author{Danish Rashid Bhat$^3$}
\address{$^{1}$Department of Mathematics, University of Kashmir, Srinagar-190006, India}
\email{$^1$dr.narather@gmail.com, $^2$Tanveerbhat054@gmail.com, $^3$danishmath1904@gmail.com}

\begin{abstract}
Let $\Re_n$ be the set of all rational functions of the type $r(z) = p(z)/w(z),$ where $p(z)$ is a polynomial of degree at most $n$ and  $w(z) = \prod_{j=1}^{n}(z-a_j)$, $|a_j|>1$ for $1\leq j\leq n$. In this paper, we set up some results for rational functions with fixed poles and restricted zeros. The obtained results bring forth generalizations and refinements of some known inequalities for rational functions and in turn produce generalizations and refinements of some polynomial inequalities as well.
\\
\smallskip
\newline
\noindent \textbf{Keywords:} Rational functions, polynomials , inequalities. \\
\noindent \textbf{Mathematics Subject Classification (2020)}: 30A10, 30C10, 30C15.
\end{abstract}

\maketitle

\section{\textbf{Introduction }}
Let $P_n$ denote the class of all complex polynomials of degree at most $n$ and let $D_{k-}= \left\{z :|z|<k\right\}$, $D_{k+}= \left\{z : |z| > k\right\}$ and $T_{k}= \left\{z : |z| = k\right\}$. For $a_j\in{\mathbb{C}}$, $j = 1, 2, \dots , n,$  we write 

\begin{align*}
w(z) := \prod_{j=1}^{n}(z -a_j),\qquad B(z):= \prod_{j=1}^{n}\left(\frac{1 - \overline{a_j}z}{z- a_j}\right)
\end{align*} 
  and
  \begin{align*} 
  \Re_n := \Re_n(a_1, a_2, \dots , a_n) = \left\{\frac{p(z)}{w(z)}; p\in P_n\right\}.
  \end{align*}
 Then $\Re_n$ is the set of all rational functions having poles possibly at $a_1, a_2,\dots,a_n$ and with finite limit at infinity. It is clear that $B(z)\in \Re_n$ and $|B(z)|$ = 1 for $z\in T_1.$ Further for any complex valued function $f$ defined on  $T_1,$ we set $ \| f \| =\sup_{z\in T_1}|f(z)|,$ the Chebshev norm of $f$ on $T_1.$ Throughout this paper, we shall assume that all the poles $a_j$, $j = 1, 2, \dots, n$ lie in $D_{1+}$.
 \\~\\
If $p\in P_n$, then concerning the estimate of $\|p^\prime\|$ in terms of $\|p\|$ on $T_1,$ we have the following famous result known as Bernstein's inequality \cite{SB}. 
  \begin{thmx}\label{tA}
 If $p\in P_n$, then
 \begin{align}\label{tA1}
 \|p^\prime\|\leq n\|p\|
 \end{align}
with equality only for $p(z) = \lambda z^n,$ $\lambda\neq 0$ being a complex number.  
  \end{thmx}
  For the class of polynomials having no zeros in $D_{1-},$ inequality \eqref{tA1} can be improved. In this direction,  Erd\"{o}s conjectured, and later Lax \cite{EL} proved the following result:
  \begin{thmx} \label{tB}
   If $p\in P_n$ having no zeros in the open unit disc $D_{1-},$ then
  \begin{align}\label{tB2}
  \|p^\prime\|\leq \frac{n}{2}\|p\|
  \end{align}
  with equality for those polynomials, which have all their zeros on $T_1.$
  \end{thmx}
  \noindent Malik \cite{MAM} considered the class of polynomials having no zeros in the circle of radius $k,$ and proved the following generalization of Theorem \ref{tB}. 
  \begin{thmx} \label{tC}
  If $p\in P_n$ having no zeros in the open unit disc $D_{k-},\ k \ge 1,$ then  
  \begin{align}\label{tC3}
  \|p^\prime\|\leq \frac{n}{1+k}\|p\|.
  \end{align}
  The result is best possible and equality holds for $p(z)=(z+k)^n.$
 \end{thmx}
\noindent Li, Mohapatra and Rodriguez \cite{LMR} extended Bernstein’s inequality \eqref{tA1} to the rational functions $r\in\Re_n$ with prescribed poles and replace $z^n$ by Blaschke product $B(z).$ Among other things they proved the following result:
  \begin{thmx}\label{tD}
   If $r\in\Re_n$, then for $z\in T_1,$
 \begin{align}\label{tD4}
  \left|r^{\prime}(z)\right| \leq |B^{\prime}(z)||r||.
  \end{align}
  Equality in \eqref{tD4} holds for $r(z)=u B(z)$ with $u\in T_1.$
    \end{thmx}
  \noindent For rational functions in $\Re_n$ with restricted zeros, they \cite{LMR} also proved the following result:
  \begin{thmx}\label{tE}
  Let $r\in \Re_{n}$ have all its zeros in $T_1 \cup D_{1+}$, then for $z\in T_1$, we have 
  \begin{align}\label{tE5}
  \left|r^{\prime}(z)\right| \leq \frac{1}{2}\left|B^{\prime}(z)\right| \|r\|.
  \end{align}
  Equality in \eqref{tE5} holds for $r(z)=u B(z) + v $ with $u,v \in T_1.$
  \end{thmx}
  \noindent As a generalization of Theorem \ref{tE}, Aziz and Zargar \cite{AZ} obtained the following result:
 \begin{thmx}\label{tF}
 Let $r\in \Re_{n}$ have all its zeros in $T_k \cup D_{k+}$, $k\geq1$, then for $z\in T_1$, we have
 \begin{align}\label{tF6}
 \left|r^{\prime}(z)\right| \leq \frac{1}{2}\left\lbrace \left|B^{\prime}(z)\right|- \frac{n(k-1)}{(k+1)} - \frac{|r(z)|^{2}}{\|r\| ^2}\right\rbrace  \|r\|.
 \end{align}
 Equality in \eqref{tF6} occurs at $z=1$ for 
 \begin{align*}
 r(z) = \left(\frac{z+k}{z-a}\right)^n ~,   \quad  B(z) = \left(\frac{1 - az}{z - a}\right)^n, \quad  a>1 ,~ k\ge 1.
 \end{align*}
 \end{thmx}
 Numerous inequalities between rational functions in both directions form an essential part of the classical content of geometric function theory. These inequalities serve as generalizations of the polynomial inequalities and play a crucial role in the field of approximation theory. Many recent articles have been published on this topic, demonstrating its importance (e.g., see \cite{NAI},\cite{MSW},\cite{MSW1},\cite{MSW2}).
\section{\textbf{Main results}}
 In this section, we  first arrive at certain generalizations and refinements of Theorems \ref{tE} and \ref{tF}. Then as an application of these results, we obtain some new and elegant generalizations and refinements of the known polynomial inequalities due to Erd\"{o}s-Lax \cite{EL} and Malik \cite{MAM}. We begin by presenting the following generalization and refinement of Theorem \ref{tF}.
 \begin{theorem}\label{TE11}   
Suppose $r\in \Re_{n}$ has all its zeros in $T_k \cup D_{k+}$, $k\geq1;$ that is $r(z)= p(z)/w(z)$ with $p(z)=c_n\prod\limits_{j=1}^{n}({z -z_j}),$ $c_n\ne0 ,  |z_j|\ge k\ge 1$. Then for all $z\in T_1$, other than the zeros of $r(z)$ and $|\beta|\leq1,$
\begin{align}\label{TE11e}
\left|\frac{zr^\prime(z)}{r(z)}+\frac{\beta}{1+k}|B^\prime(z)|\right|\leq \frac{1}{2} \left\lbrace|B^\prime(z)|- \frac{n|r(z)|^{2}}{\|r\|^{2}}\frac{(k-1)}{(k+1)}- \frac{2|r(z)|^{2}}{\|r\|^{2}}\left(\frac{n}{k+1} - \sum\limits_{j=1}^{n} \frac{1}{1+|z_{j}|} - \frac{Re(\beta)}{(1+k)}|B^{\prime}(z)| \right)\right\rbrace\frac{\|r\|}{|r(z)|}.
 \end{align}
 Inequality \ref{TE11e} is sharp in the case $\beta=0$ and equality holds for
 \begin{align*}
 r(z) = \left(\frac{z+k}{z-a}\right)^n ~,   \quad  B(z) = \left(\frac{1 - az}{z - a}\right)^n, \quad at \quad z=1, \quad  a>1, \quad k\ge1 \quad and \quad \beta=0.
 \end{align*}
 \end{theorem}
 \begin{remark}
 \textnormal{In the case $k=1,$ Theorem \ref{TE11} immediately yields to the following generalization and refinement of Theorem \ref{tE}.}
 \end{remark}
  \begin{corollary}\label{t1}
  Suppose $r\in \Re_{n}$ has all its zeros in $T_1 \cup D_{1+};$ that is $r(z)= p(z)/w(z)$ with $p(z)=c_n\prod\limits_{j=1}^{n}({z -z_j}),$ $c_n\ne0 ,  |z_j|\ge 1$. Then for all $z\in T_1$, other than the zeros of $r(z)$ and $|\beta|\leq1,$
\begin{align}\label{t1e}
\left|\frac{zr^\prime(z)}{r(z)} +\frac{\beta}{2}|B^\prime(z)|\right|\leq \frac{1}{2} \left\lbrace|B^\prime(z)|- \frac{2|r(z)|^{2}}{\|r\|^{2}}\left(\frac{n}{2} - \sum_{j=1}^{n} \frac{1}{1+|z_{j}|} - \frac{Re(\beta)}{2}|B^{\prime}(z)| \right)\right\rbrace\frac{\|r\|}{|r(z)|}.
 \end{align}
 Inequality \ref{t1e} is sharp in the case $\beta=0$ and equality holds for
 \begin{align*}
 r(z) = \left(\frac{z+1}{z-a}\right)^n ~,   \quad  B(z) = \left(\frac{1 - az}{z - a}\right)^n, \quad at \quad z=1, \quad  a>1 \quad and \quad \beta=0.
 \end{align*}
 \end{corollary}
 \begin{remark}
 \textnormal{Although Theorem \ref{TE11} generalizes and improves Theorem \ref{tF}, but the bound given in \eqref{TE11e} demands that all the zeros of $r(z)$ be known beforehand. Since the computation of zeros is not always an easy piece of work, therefore in those situations where the zeros of $r(z)$ are unspecified, one may desire to have a result which instead of the zeros of $r(z)$ depends upon some coefficients of $p(z)$ in $r(z)=p(z)/w(z)$. Our next result serves this purpose.}
 \end{remark}
 \begin{corollary}\label{t2}
  Suppose $r\in \Re_{n}$ has all its zeros in $T_k \cup D_{k+}$, $k\geq1;$ that is $r(z)= p(z)/w(z)$ with $p(z)=c_n\prod\limits_{j=1}^{n}({z -z_j})=c_0 +c_1z+\dots+c_nz^n,$ $c_n\ne0 , |z_j|\ge k\ge 1$.Then for all $z\in T_1$, other than the zeros of $r(z)$ and $|\beta|\leq1,$
  \begin{align}\label{t2e}
\left|\frac{zr^\prime(z)}{r(z)}+\frac{\beta}{1+k}|B^\prime(z)|\right|\leq \frac{1}{2} \left\lbrace|B^\prime(z)|- \frac{n|r(z)|^{2}}{\|r\|^{2}}\frac{(k-1)}{(k+1)}- \frac{2|r(z)|^{2}}{(k+1)\|r\|^{2}}\left(\frac{|c_{0}|-k^{n}|c_{n}|}{|c_{0}|+k^{n}|c_{n}|} - Re(\beta) \ |B^{\prime}(z)| \right)\right\rbrace\frac{\|r\|}{|r(z)|}.
\end{align}
Inequality \ref{t2e} is sharp in the case $\beta=0$ and equality holds for
 \begin{align*}
 r(z) = \left(\frac{z+k}{z-a}\right)^n ~,   \quad  B(z) = \left(\frac{1 - az}{z - a}\right)^n, \quad at \quad z=1, \quad  a>1, \quad k\ge1 \quad and \quad \beta=0.
 \end{align*}
\end{corollary}
\begin{proof}[\bf Proof of Corollary \ref{t2}]
In view of inequality $\ref{TE11e}$, we have
\begin{align}\label{t2ee}
\left|\frac{zr^\prime(z)}{r(z)}+\frac{\beta}{1+k}|B^\prime(z)|\right| \leq \nonumber \frac{1}{2} \left\lbrace|B^\prime(z)|- \frac{n|r(z)|^{2}}{\|r\|^{2}}\frac{(k-1)}{(k+1)}- \frac{2|r(z)|^{2}}{\|r\|^{2}}\left(\frac{n}{(k+1)} - \sum_{j=1}^{n} \frac{1}{1+|z_{j}|} - \frac{Re(\beta)}{(1+k)}|B^{\prime}(z)| \right)\right\rbrace\frac{\|r\|}{|r(z)|}  & \\ = \nonumber \frac{1}{2} \left\lbrace|B^\prime(z)|- \frac{n|r(z)|^{2}}{\|r\|^{2}}\frac{(k-1)}{(k+1)}- \frac{2|r(z)|^{2}}{\|r\|^{2}}\left(-\frac{1}{k+1}\sum_{j=1}^{n} \frac{k-|z_{j}|}{1+|z_{j}|} - \frac{Re(\beta)}{(1+k)}|B^{\prime}(z)| \right)\right\rbrace\frac{\|r\|}{|r(z)|} & \\ \leq \nonumber \frac{1}{2} \left\lbrace|B^\prime(z)|- \frac{n|r(z)|^{2}}{\|r\|^{2}}\frac{(k-1)}{(k+1)}- \frac{2|r(z)|^{2}}{\|r\|^{2}}\left(-\frac{1}{k+1}\sum_{j=1}^{n} \frac{k-|z_{j}|}{k+|z_{j}|} - \frac{Re(\beta)}{(1+k)}|B^{\prime}(z)| \right)\right\rbrace\frac{\|r\|}{|r(z)|} & \\
\leq \frac{1}{2} \left\lbrace|B^\prime(z)|- \frac{n|r(z)|^{2}}{\|r\|^{2}}\frac{(k-1)}{(k+1)} - \frac{2|r(z)|^{2}}{\|r\|^{2}}\left(-\frac{1}{k+1}\sum_{j=1}^{n} \frac{1-\frac{|z_{j}|}{k}}{1+\frac{|z_{j}|}{k}} - \frac{Re(\beta)}{(1+k)}|B^{\prime}(z)| \right)\right\rbrace\frac{\|r\|}{|r(z)|}.
 \end{align}
 Using Lemma \ref{l4} in inequality \eqref{t2ee} and noting that $\frac{|z_{j}|}{k}\geq1$, $j=1,2,\cdots,n$, we obtain
 \begin{align}
\left|\frac{zr^\prime(z)}{r(z)}+\frac{\beta}{1+k}|B^\prime(z)|\right| & \leq \nonumber \frac{1}{2} \left\lbrace|B^\prime(z)|- \frac{n|r(z)|^{2}}{||r||^{2}}\frac{(k-1)}{(k+1)}- \frac{2|r(z)|^{2}}{\|r\|^{2}}\left(-\frac{1}{k+1} \times \frac{1-\prod\limits_{j=1}^{n}\frac{|z_{j}|}{k}}{1+\prod\limits_{j=1}^{n}\frac{|z_{j}|}{k}} - \frac{Re(\beta)}{(1+k)}|B^{\prime}(z)| \right)\right\rbrace\frac{\|r\|}{|r(z)|} \\ & = \nonumber \frac{1}{2} \left\lbrace|B^\prime(z)|- \frac{n|r(z)|^{2}}{\|r\|^{2}}\frac{(k-1)}{(k+1)}- \frac{2|r(z)|^{2}}{\|r\|^{2}}\left( -\frac{1}{k+1} \times \frac{k^{n}|c_{n}|- |c_{0}|}{k^{n}|c_{n}|+|c_{0}|} - \frac{Re(\beta)}{(1+k)}|B^{\prime}(z)| \right)\right\rbrace\frac{\|r\|}{|r(z)|} \\ &  = \nonumber \frac{1}{2} \left\lbrace|B^\prime(z)|- \frac{n|r(z)|^{2}}{\|r\|^{2}}\frac{(k-1)}{(k+1)}- \frac{2|r(z)|^{2}}{(k+1)\|r\|^{2}}\left(\frac{|c_{0}|-k^{n}|c_{n}|}{|c_{0}|+k^{n}|c_{n}|} - Re(\beta)\ |B^{\prime}(z)| \right)\right\rbrace\frac{\|r\|}{|r(z)|}.
\end{align}
\end{proof}
 This completes the proof of corollary \ref{t2}.
 \begin{remark}
 \textnormal{For $k=1,$ Corollary \ref{t2} yields the following generalization and refinement of Theorem \ref{tE}.}
 \end{remark}
 \begin{corollary}\label{t3}
 Suppose $r\in \Re_{n}$ has all its zeros in $T_1 \cup D_{1+};$ that is $r(z)= p(z)/w(z)$ with $p(z)=c_n\prod\limits_{j=1}^{n}({z -z_j})=c_0 +c_1z+\dots +c_nz^n,$ $c_n\ne0 ,  |z_j|\ge 1$. Then for all $z\in T_1$, other than the zeros of $r(z)$ and $|\beta|\leq1,$ 
\begin{align}\label{t3e}
\left|\frac{zr^\prime(z)}{r(z)} +\frac{\beta}{2}|B^\prime(z)|\right|\leq \frac{1}{2} \left\lbrace|B^\prime(z)|- \frac{|r(z)|^{2}}{\|r\|^{2}}\left(\frac{|c_{0}|-|c_{n}|}{|c_{0}|+|c_{n}|} - Re(\beta)\ |B^{\prime}(z)| \right)\right\rbrace\frac{\|r\|}{|r(z)|}.
 \end{align}
 Inequality \ref{t3e} is sharp in the case $\beta=0$ and equality holds for
 \begin{align*}
 r(z) = \left(\frac{z+1}{z-a}\right)^n ~,   \quad  B(z) = \left(\frac{1 - az}{z - a}\right)^n, \quad at \quad z=1, \quad  a>1 \quad and \quad \beta=0.
 \end{align*}
 \end{corollary}
 As an application of Theorem \ref{tA}, we next present the following generalization and refinement of Theorem \ref{tC} concerning polynomials not vanishing in the open disk $D_{k-},\ k \ge 1.$
\begin{corollary}\label{t4}
  If $p\in P_n$ having no zeros in the open unit disc $D_{k-},\ k \ge 1,$ i,e $p(z)=c_{n}\prod\limits_{j=1}^{n}(z-z_{j})=c_{0}+c_{1}z+\cdots+c_{n}z^{n}$, $c_{n}\neq0$, $|z_{j}|\geq k\geq1$, $j=1,2,\cdots,n$. Then for $z\in T_1$ and $|\beta|\leq1$,   
 \begin{align*}
\left|zp^\prime(z) + \frac{n \beta}{1+k} p(z)\right|\leq \frac{\|p\|}{2}\left\lbrace n-\frac{2}{(k+1)}\left(\frac{n(k-1)}{2}+\frac{|c_{0}|-k^{n}|c_{n}|}{|c_{0}+k^{n}|c_{n}|}-n \ Re(\beta) \right)\frac{|p(z)|^{2}}{\|p\|^{2}}\right\rbrace.
 \end{align*}
 \begin{proof} [\bf Proof of Corollary \ref{t4}]
 Taking $r(z)=\frac{p(z)}{w(z)}$ with $w(z)=(z-\alpha)^{n}$ $\alpha >k\geq1$, so that $B(z)=\left(\frac{1-\alpha z}{z-\alpha}\right)^{n},$ from corollary \ref{t2}, we obtain
\begin{align}\label{t4e}
\left|\frac{zr^\prime(z)}{r(z)}+\frac{\beta}{1+k}|B^\prime(z)|\right| & \leq \nonumber \frac{1}{2} \left\lbrace|B^\prime(z)|- \frac{n|r(z)|^{2}}{\|r\|^{2}}\frac{(k-1)}{(k+1)}- \frac{2|r(z)|^{2}}{(k+1)\|r\|^{2}}\left(\frac{|c_{0}|-k^{n}|c_{n}|}{|c_{0}|+k^{n}|c_{n}|} - Re(\beta) \ |B^{\prime}(z)| \right)\right\rbrace\frac{\|r\|}{|r(z)|} \\ & =\frac{1}{2}\left\lbrace|B^{\prime}(z)|-\frac{2}{(k+1)}\left(\frac{n(k-1)}{2}+ \frac{|c_{0}|-k^{n}|c_{n}|}{|c_{0}|+k^{n}|c_{n}|} -Re(\beta) \ |B^{\prime}(z)| \right)\frac{|r(z)|^{2}}{\|r\|^2} \right\rbrace\frac{\|r\|}{|r(z)|}.
\end{align}
Assuming $\| r \|=|r(e^{i\theta})|=|\frac{p(e^{i\psi_{0}}}{w(e^{i\psi_{0}}}|$, for $\psi_{0}\in [0,2\pi)$, inequality \eqref{t4e} can be written as
\begin{align*}
\left|z\left(\frac{p^{\prime}(z)}{p(z)}-\frac{n}{(z-\alpha)}\right)+ \frac{\beta}{1+k}|B^{\prime}(z)|\right| 
\end{align*}
\begin{align}\label{t4ee}
\leq  \frac{1}{2}\left\lbrace|B^{\prime}(z)|-\frac{2}{(k+1)}\left(\frac{n(k-1)}{2} + \frac{|c_{0}-k^{n}|c_{n}|}{|c_{0}|+k^{n}|c_{n}|}-  Re(\beta)\ |B^{\prime}(z)|\right)\frac{|p(z)|^{2}}{|p(e^{i\psi_{0}})|^{2}} \left|\frac{e^{i\psi_{0}}-\alpha}{z- \alpha}\right|^{2n} \right\rbrace\frac{|p(e^{i\psi_{0}}|}{|e^{i\psi_{0}} -\alpha|^{n}}\frac{|z-\alpha|^n}{|p(z)|}.
\end{align}
Since this is true for every $\alpha >1$, letting $\alpha \to \infty$ in inequality \eqref{t4ee} and noting that $|B^{\prime}(z)|\rightarrow |nz|=n$ as $\alpha \rightarrow \infty$ for $z\in T_1$, we obtain
\begin{align*}
\left|zp^\prime(z) + \frac{n \beta}{1+k} p(z)\right| & \leq  \frac{1}{2}\left\lbrace|B^{\prime}(z)|-\frac{2}{(k+1)}\left(\frac{n(k-1)}{2} + \frac{|c_{0}-k^{n}|c_{n}|}{|c_{0}|+k^{n}|c_{n}|}-  Re(\beta)\ |B^{\prime}(z)|\right)\frac{|p(z)|^{2}}{|p(e^{i\psi_{0}})|^{2}} \right\rbrace |p(e^{i\psi_{0}}| \\ & \leq \frac{\|p \|}{2}\left\lbrace n-\frac{2}{(k+1)}\left(\frac{n(k-1)}{2} + \frac{|c_{0}-k^{n}|c_{n}|}{|c_{0}|+k^{n}|c_{n}|}- Re(\beta) \ |B^{\prime}(z)|\right)\frac{|p(z)|^{2}}{\|p\|^{2}}\right\rbrace.
\end{align*}
This completes the proof of corollary \ref{t4}.
 \end{proof}
 \end{corollary}
 \begin{remark}
 \textnormal{For $k=1,$ Corollary \ref{t4} yields the following generalization and refinement of Theorem \ref{tB}.}
 \end{remark}
\begin{corollary}\label{t5}
  If $p\in P_n$ having no zeros in the open unit disc $D_{1-},$ i,e $p(z)=c_{n}\prod\limits_{j=1}^{n}(z-z_{j})=c_{0}+c_{1}z+\cdots + c_{n}z^{n}$, $c_{n}\neq0$, $|z_{j}|\geq 1$, $j=1,2,\cdots,n$. Then for $z\in T_1$ and $|\beta|\leq1$,
\begin{align*}
\left|zp^\prime(z) + \frac{n \beta}{2} p(z)\right|\leq \frac{\|p\|}{2}\left\lbrace n-\left(\frac{|c_{0}|-|c_{n}|}{|c_{0}+|c_{n}|}-n \ Re(\beta) \right)\frac{|p(z)|^{2}}{\|p\|^{2}}\right\rbrace.
\end{align*}
\\~\\
 \end{corollary}
 \section{Lemmas}
\noindent For the proof of our results, we need the following lemmas. The first lemma is due to Aziz and Zargar \cite{AZ}.
\begin{lemma}\label{l1}
If $z\in T_1$, then
\begin{align*}
Re\left(\frac{zw^\prime(z)}{w(z)}\right) = \frac{n - |B^\prime(z)|}{2},
\end{align*} 
 where $w(z) = \prod\limits_{j=1}^{n}(z -a_j)$. 
\end{lemma}
The following two lemmas are due to Li,Mohapatra and Rodrigues \cite{LMR}.
\begin{lemma}\label{l2}
If $z\in T_1$, then
\begin{align*}
z\frac{B^{\prime}(z)}{B(z)}=\left|B^\prime(z)\right|.
\end{align*}
\end{lemma}
\begin{lemma}\label{l3}
If $r\in \Re_{n}$ and $r^{*}(z)=B(z)\overline{r\left(\frac{1}{\overline{z}}\right)}$, then for $z\in T_1$, we have
\begin{align*}
\left|\left(r^{*}(z)\right)^{\prime}\right|+\left|r^{\prime}(z)\right|\leq \left|B^{\prime}(z)\right| \|r\|.
\end{align*}
\end{lemma}
\begin{lemma}\label{l4}
If $\langle \zeta_j \rangle_{j=1}^\infty$ be a sequence of real numbers such that $\zeta_j\ge 1 $, $j \in \mathbb{N},$ then
\begin{align*}
\sum_{j=1}^n \frac{1- \zeta_j}{ 1+\zeta_j}\leq \frac{1- \prod\limits_{j=1}^{n}  \zeta_j}{1+ \prod\limits_{j=1}^{n}  \zeta_j }\ \ \forall \  n\in\mathbb{N}.
\end{align*}
\begin{proof}
The claim follows easily by induction on $n$. Hence, we omit details.
\end{proof}
\end{lemma}
The last lemma which we need is due to Aziz and Shah \cite{AS}.
\begin{lemma}\label{l5}
If $r\in \Re_{n}$ and all the $n$ zeros of $r$ lie in $T_k \cup D_{k-}$, $k\leq 1$, then for $z\in T_1$,
\begin{align*}
\left|r^{\prime}(z)\right|\geq \frac{1}{2}\left\lbrace\left|B^{\prime}(z)\right|+\frac{n(1-k)}{(1+k)}\right\rbrace\left|r(z)\right|.
\end{align*}
\end{lemma}
\section{Proofs of the theorems}
\begin{proof}[\bf Proof of Theorem \ref{TE11}]
Since $r\in \Re_{n}$ and all the zeros of $r(z)$ lie in $T_k \cup D_{k+}$, $k\geq 1$. Let $z_{1},z_{2},\dots,z_{n}$ be the $n$ zeros of $r(z) = \frac{p(z)}{w(z)}$, where $p(z)=c_{n}\prod\limits_{j=1}^{n}(z-z_{j})$; $|z_{j}|\geq k \geq1, j=1,2,\dots,n$.
Then for $|\beta|\leq1$ and for all $z\in T_1$, where $r(z)\ne 0$, we have 
\begin{align*}
Re\left\{\frac{zr^\prime(z)}{r(z)} + \frac{\beta}{1+k}|B^\prime(z)|\right\} & = Re\left\{\frac{zr^\prime(z)}{r(z)}\right\} + \frac{|B^\prime(z)|}{1+k}Re\left\{\beta\right\}\\ & = Re\left\{\frac{zp^\prime(z)}{p(z)} - \frac{zw^\prime(z)}{w(z)}\right\} + \frac{|B^\prime(z)|}{1+k}Re\left\{\beta\right\} \\ & = Re\left\{\frac{zp^\prime(z)}{p(z)}\right\} - Re\left\{\frac{zw^\prime(z)}{w(z)}\right\} + \frac{|B^\prime(z)|}{1+k}Re\left\{\beta\right\}.
\end{align*}
On using Lemma \ref{l1}, we have for $|\beta|\leq 1$ and for all $z\in T_1$, where $r(z)\ne 0$,
\begin{align}\nonumber\label{14}
Re\left\{\frac{zr^\prime(z)}{r(z)} + \frac{\beta}{1+k}|B^\prime(z)|\right\}& = Re\sum_{j=1}^{n}\left\{\frac{z}{z - z_j}\right\} -\left\{\frac{n -|B^\prime(z)}{2}\right\} + \frac{|B^\prime(z)|}{1+k}Re\left\{\beta\right\}  \\ & = \sum_{j=1}^{n}Re\left\{\frac{z}{z - z_j}\right\} -\frac{n}{2} + \frac{1}{2}\left\{1 + \frac{2Re(\beta)}{1+k}\right\}|B^\prime (z)|.
\end{align}
It can be easily verified that for $z\in T_1$ and $|z_j|\geq k\geq 1$,
\begin{align*}
Re\left\{\frac{z}{z - z_j}\right\}\leq \frac{1}{1 + |z_j|}. 
\end{align*}
Using this in inequality \eqref{14}, we get for $|\beta|\leq 1$ and for all $z \in T_1,     $ where $r(z)\ne 0$,
\begin{align}\label{15}
Re\left\{\frac{zr^\prime(z)}{r(z)} + \frac{\beta}{1+k}|B^\prime(z)|\right\} \leq \sum_{j=1}^{n}\left\{\frac{1}{1 +  |z_j|}\right\} - \frac{n}{2} + \frac{1}{2}\left\{1 + \frac{2Re(\beta)}{1+k}\right\}|B^\prime (z)|. 
\end{align}
Now if $r^{*}(z)=B(z)\overline{r(\frac{1}{\overline{z}})}$, then for $|\beta|\leq1$ and for all $z\in T_1$, where $r(z)\ne 0$, we have
\begin{align}\nonumber \label{16}
\left|\frac{(r^{*}(z))^{\prime}}{\overline{r(z)}}-\frac{\overline{\beta}}{1+k}|B^{\prime}(z)|\right|^{2} & =\left|\frac{1}{\overline{r(z)}}\left(zB^{\prime}(z)\overline{r(z)}-B(z)\overline{zr^{\prime}(z)}\right)-\frac{z\overline{\beta}B^{\prime}(z)}{(1+k)}\right|^{2} \\ &
=\left|\frac{zB^{\prime}(z)}{B(z)}\frac{\overline{r(z)}}{\overline{r(z)}}-\frac{B(z)\overline{zr^{\prime}(z)}}{B(z)\overline{r(z)}}-\frac{z\overline{\beta}B^{\prime}(z)}{B(z)(1+k)}\right|^{2} \ \  {\because |B(z)|=1 \ \textnormal{for} \ z\in T_1.}
\end{align}
In view of lemma \ref{l2}, equation \ref{16} gives for $|\beta|\leq1$ and for all $z\in T_1,$ where $r(z)\ne 0$,
\begin{align}\label{17}
\left|\frac{(r^{*}(z))^{\prime}}{\overline{r(z)}}-\frac{\overline{\beta}}{1+k}|B^{\prime}(z)|\right|^{2} & \nonumber = \left||B^{\prime}(z)|- \left(\frac{\overline{zr^{\prime}(z)}}{\overline{r(z)}}+ \frac{|B^{\prime}(z)|\ \overline{\beta}}{(1+k)}\right)\right|^{2} \\ & \nonumber
=\left||B^{\prime}(z)|- \left(\frac{zr^{\prime}(z)}{r(z)}+ \frac{|B^{\prime}(z)| \ \beta}{(1+k)}\right)\right|^{2} \\& 
=|B^{\prime}(z)|^{2}+ \left|\frac{zr^{\prime}(z)}{r(z)}+ \frac{|B^{\prime}(z)| \ \beta}{(1+k)}\right|^{2}-2|B^{\prime}(z)|\ Re\left\lbrace\frac{zr^{\prime}(z)}{r(z)}+|B^{\prime}(z)|\frac{\beta}{1+k}\right\rbrace.
\end{align}
Using inequality \eqref{15} in \eqref{17}, we get for $|\beta|\leq1$ and for all $z\in T_1$, where $r(z)\ne 0$,
\begin{align*}\nonumber
\left|\frac{(r^{*}(z))^{\prime}}{\overline{r(z)}}-\frac{\overline{\beta}}{1+k}|B^{\prime}(z)|\right|^{2} & \geq |B^{\prime}(z)|^{2}+ \left|\frac{zr^{\prime}(z)}{r(z)}+ \frac{|B^{\prime}(z)| \ \beta}{(1+k)}\right|^{2}-2|B^{\prime}(z)|\left\lbrace\sum\limits_{j=1}^{n}\frac{1}{1+|z_{j}|}-\frac{n}{2}+\frac{1}{2}\left\lbrace1+\frac{2Re(\beta)}{1+k}\right\rbrace|B^{\prime}(z)|\right\rbrace \\&
= \left|\frac{zr^{\prime}(z)}{r(z)}+ \frac{|B^{\prime}(z)| \ \beta}{(1+k)}\right|^{2}-2|B^{\prime}(z)|\left\lbrace\sum\limits_{j=1}^{n}\frac{1}{1+|z_{j}|}-\frac{n}{2}+\frac{Re(\beta)}{1+k}|B^{\prime}(z)|\right\rbrace.
\end{align*}
Equivalently for $|\beta|\leq1$ and for all $z\in T_1$, where $r(z)\ne 0$, we have
\begin{align}\label{18}
\left\lbrace\left|\frac{z(r^{\prime}(z)}{r(z)}+\frac{\beta}{1+k}|B^{\prime}(z)|\right|^{2}-2|B^{\prime}(z)|\left\lbrace\sum\limits_{j=1}^{n}\frac{1}{1+|z_{j}|}-\frac{n}{2}+\frac{Re(\beta)}{1+k}|B^{\prime}(z)|\right\rbrace\right\rbrace^{\frac{1}{2}}
\leq\left|\frac{\left(r^{*}(z)\right)^{\prime}}{\overline{r(z)}}
-\overline{\beta}\frac{B^{\prime}(z)}{1+k}\right|.
\end{align}
Since all the zeros of $r(z)$ lie in $T_k \cup D_{k+} \ , k \geq 1$, so all the zeros of $r^{*}(z)$ lie in $T_k \cup D_{ \frac{1}{k}-}$; $\frac{1}{k}\leq1$. Thus by applying lemma \ref{l5} to $r^{*}(z)$, we obtain for $z \in T_1$, where $r(z)\ne 0$,
\begin{align}
\left|\left(r^{*}(z)\right)^{\prime}\right| & \geq \nonumber \frac{1}{2}\left\lbrace|B^{\prime}(z)|+\frac{n\left(1-\frac{1}{k}\right)}{1+\frac{1}{k}}\right\rbrace|r^{*}(z)| \\ & 
\geq \nonumber \frac{|B^{\prime}(z)|}{k+1} \ |r^{*}(z)| \ \ \  \textnormal{for}\ k \ge 1.   
\end{align}
This implies for $|\beta|\leq1$ and using the fact $|r^{*}(z)|=|r(z)|$ for $z\in T_1,$ where $r(z)\ne 0$,
\begin{align}\nonumber
\left|\left(r^{*}(z)\right)^{\prime}\right|\geq \frac{|\beta||B^{\prime}(z)|}{k+1} \ |r(z)|.
\end{align}
In view of this, choosing the argument of $\beta$ in the right hand side of \eqref{18} such that
\begin{align}\nonumber
\left|\frac{\left(r^{*}(z)\right)^{\prime}}{\overline{r(z)}}-\overline{\beta}\frac{B^{\prime}(z)}{1+k}\right|=\left|\frac{\left(r^{*}(z)\right)^{\prime}}{r(z)}\right|- |\beta|\ \frac{|B^{\prime}(z)|}{1+k},
\end{align}
 we obtain from inequality \eqref{18}, for $|\beta|\leq1$ and for all $z\in T_1,$ where $r(z)\ne 0$, 
\begin{align*}\nonumber
\left\lbrace\left|\frac{zr^{\prime}(z)}{r(z)}+\frac{\beta}{1+k}|B^{\prime}(z)|\right|^{2}-2|B^{\prime}(z)|\left\lbrace\sum\limits_{j=1}^{n}\frac{1}{1+|z_{j}|}-\frac{n}{2}+\frac{Re(\beta)}{1+k}|B^{\prime}(z)|\right\rbrace\right\rbrace^{\frac{1}{2}} \leq \left|\frac{\left(r^{*}(z)\right)^{\prime}}{r(z)}\right|- |\beta|\ \frac{|B^{\prime}(z)|}{1+k}.
\end{align*}
This in combination with lemma \ref{l3} gives
\begin{align*}\nonumber
|r^{\prime}(z)|+|r(z)|\left\lbrace\left|\frac{z(r^{\prime}(z)}{r(z)}+\frac{\beta}{1+k}|B^{\prime}(z)|\right|^{2}-2|B^{\prime}(z)|\left( \sum_{j=1}^{m}\frac{1}{1+|z_{j}|}-\frac{n}{2}+\frac{Re{\beta}}{1+k}|B^{\prime}(z)|\right)\right\rbrace^{\frac{1}{2}}
\end{align*}
\begin{align}\nonumber
\leq \|r\| \ |B^{\prime}(z)|-\frac{|\beta|}{1+k}|B^{\prime}(z)|\ |r(z)|.
\end{align}
This gives for $|\beta|\leq1$ and for all $z\in T_1,$ where $r(z)\ne 0$,
\begin{align*}
\left\lbrace\left|\frac{z(r^{\prime}(z)}{r(z)}+\frac{\beta}{1+k}|B^{\prime}(z)|\right|^{2}-2|B^{\prime}(z)|\left( \sum_{j=1}^{n}\frac{1}{1+|z_{j}|}-\frac{n}{2}+\frac{Re{\beta}}{1+k}|B^{\prime}(z)|\right)\right\rbrace^{\frac{1}{2}} & \leq \frac{||r||}{|r(z)|}
|B^{\prime}(z)|-\left(\left|\frac{zr^{\prime}(z)}{r(z)}\right|+\frac{|\beta|}{1+k}|B^{\prime}(z)|\right)\\ &
\leq \frac{\|r\|}{|r(z)|}|B^{\prime}(z)|-\left|\frac{zr^{\prime}(z)}{r(z)}+\frac{\beta}{1+k}|B^{\prime}(z)|\right|.
\end{align*}
Which, after straightforward computation yields,
\begin{align*}
\left|\frac{z(r^{\prime}(z)}{r(z)}+\frac{\beta}{1+k}|B^{\prime}(z)|\right|\leq \frac{\|r(z)\|}{2|r(z)|}\left\lbrace|B^{\prime}(z)|-\frac{n|r(z)|^{2}}{\|r\|^{2}}+\frac{2|r(z)|^{2}}{\|r\|}\left(\sum\limits_{j=1}^{n}\frac{1}{1+|z_{j}|}+\frac{Re(\beta)}{1+k}|B^{\prime}(z)|\right)\right\rbrace.
\end{align*}
The above inequality is equivalent to
\begin{align*}
\begin{split}
 \left|\frac{z(r^{\prime}(z)}{r(z)}+\frac{\beta}{1+k}|B^{\prime}(z)|\right| { } & \leq \frac{\|r(z)\|}{2|r(z)|}\bigg\{|B^{\prime}(z)|-\frac{n|r(z)|^{2}}{\|r\|^{2}}\left(1-\frac{2}{k+1}\right)-\frac{n|r(z)|^{2}}{\|r\|^{2}}\frac{2}{k+1} \\ 
 & +\frac{2|r(z)|^{2}}{\|r\|^{2}}\left(\sum\limits_{j=1}^{n}\frac{1}{1+|z_{j}|}+\frac{Re(\beta)}{1+k}|B^{\prime}(z)|\right)\bigg\}.
 \end{split}
 \end{align*}
 This gives for $|\beta|\leq 1$ and for all $z\in T_1,$ where $r(z)\ne 0$,
 \begin{align*}
 \left|\frac{z(r^{\prime}(z)}{r(z)}+\frac{\beta}{1+k}|B^{\prime}(z)|\right|\leq \frac{||r(z)||}{2|r(z)|}\left\lbrace|B^{\prime}(z)|-\frac{n|r(z)|^{2}}{||r||^{2}}\frac{(k-1)}{(k+1)} -\frac{2|r(z)|^{2}}{\|r\|^{2}}\left(\frac{n}{k+1}-\sum\limits_{j=1}^{n}\frac{1}{1+|z_{j}|}-\frac{Re(\beta)}{k+1}|B^{\prime}(z)\right)\right\rbrace.
 \end{align*}
This completes the proof of Theorem $\ref{TE11}$.
\end{proof}

\end{document}